\documentclass{amsart}%%%%%%%%%%%%{article}
\usepackage {amssymb}
\usepackage {amsmath}
\usepackage{amsthm}
\usepackage{graphicx}
\usepackage {amscd}
\usepackage {epic}
\usepackage {color}
\usepackage[alphabetic]{amsrefs}
\usepackage[colorlinks, citecolor=blue]{hyperref}
\usepackage[all]{xy}
\newcommand{\E}{{\mathbb{E}}}

\newtheorem{thm}{Theorem}[section]

\newtheorem{lem}[thm]{Lemma}
\newtheorem{cor}[thm]{Corollary}

\newtheorem{prop}[thm]{Proposition}

\theoremstyle{definition}
\newtheorem{defn}[thm]{Definition}

\newtheorem{rem}[thm]{Remark}

\newtheorem{defn-thm}[thm]{Definition--Theorem}  %!!!!!!!!!!!!!!!!!!!!!!!!
\newtheorem{defn-lem}[thm]{Definition--Lemma}  %!!!!!!!!!!!!!!!!!!!!!!!!
\theoremstyle{remark}

\newcommand{\QQ}{\mathbb{Q}}

\input{diagrams.sty}

\begin{document}
\title{On base point free theorem of threefolds in positive characteristic}
\begin{abstract}{
Let  $(X,\Delta)$ be a projective klt three dimensional pair defined over  an algebraically closed field $k$ with ${\rm char} (k)>5$.  Let $L$ be a nef and big line bundle on $X$ such that $L-K_X-\Delta$ is big and nef. We show that $L$ is indeed semi-ample.}
\end{abstract}

%\author    {J\'anos Koll\'ar}
%\address   {Department of Mathematics,
%            Princeton University, Princeton NJ 08544-1000}
%\email     {kollar@math.princeton.edu}
%
\author    {Chenyang Xu}
\address   {Beijing International Center of Mathematical Research,
          5  Yiheyuan Road, Beijing, 100871, China} 
\email     {cyxu@math.pku.edu.cn}

\date{\today}

\maketitle{}

\tableofcontents
\section{Introduction}Throughout the paper, the ground field will be an algebraically closed field $k$ of characteristic $p>0$.
The main purpose of this paper is to prove the following theorem.
\begin{thm}\label{t-main1}
Assume $k$ to be an algebraically closed field of characteristic $p>5$.  Let $(X,\Delta)$ be a 3-dimensional  klt pair which is projective over a quasi-projective variety $U$.
Assume $L$ is a relatively big and nef $\mathbb{Q}$-divisor such that $K_X+\Delta+L$  is big and nef over $U$. Then $K_X+\Delta+L$ is semi-ample over $U$. 
\end{thm}

%For simplicity, in the discussion below we assume $U={\rm Spec}(k)$. 
When $k= \overline{\mathbb{F}}_p$, this is proved by Keel in \cite{Keel99}. In fact, Keel proved in general $K_X+\Delta+L$ is {\it endowed with a map} (EWM), i.e., there exists a morphism $f:X\to Z$ to an algebraic space $Z$, such that a curve vertical over $U$ is contracted by $f$ if and only if its intersection with $K_X+\Delta+L$ is 0. Furthermore, to check $K_X+\Delta+L$ is semi-ample over $U$, it suffices to show that the restriction $(K_X+\Delta+L)|_{{\rm Ex}(f)}$ is semi-ample over  $U$ where
 ${\rm Ex}(f)$ is the exceptional locus  of $f:X\to Z$. 
%The aim of this note is to verify the following theorem which gives conditions for $Z$ to be a projective variety. 

%\begin{thm}\label{t-main2}
%Let $(X,\Delta)$ be a klt three dimensional pair, such that $L=K_X+\Delta$ is nef and $\mathbb{E}(L)$ is  of 1-dimensional. Then $K_X+\Delta$ is base point free. 
%E\end{thm}

Besides using Keel's results, our approach also relies heavily on the recent results on minimal model program (MMP) in dimension three in positive characteristics (see \cite{HX13,Birkar13}). More precisely, when $p>5$, combining with Shokurov's idea (cf. \cite{Shokurov92}) of reduction to special MMP and the recent development of lifting sections coming from Frobenius image initiated in \cite{HH90} (see \cite{Schwede11} and references therein),
 it is proved that an MMP sequence  can be run in a generalized sense in \cite{HX13}  (cf. Section \ref{ss-gmmp}), where the coefficients are also assumed to be contained in the standard  set 
$$\{\frac{n-1}{n}|n\in \mathbb{N}\}.$$ 
Later,  using Shokurov's reduction technique again, \cite{Birkar13}  removes the restriction on the coefficients by reducing the general case to the one in \cite{HX13}.  Using this existence of a sequence of generalized MMP, after passing to an \'etale covering of  the algebraic space $Z$ given by \cite{Keel99}, a standard trick by running MMP can change the model $X$ over $Z$, such that  the exceptional locus ${\rm Ex}(f)$ is contained in $\lfloor \Delta \rfloor$ for a plt pair $K_X+\Delta$.  Then Theorem \ref{t-main1} follows from Keel's theorem in the semi-ample case and the abundance for surfaces.

Now we want to apply Theorem \ref{t-main1} to get the standard consequences on birational contractions when we run MMP for a klt pair $(X,\Delta)$. Let $R$ be a $(K_X+\Delta)$-negative extremal ray over $U$, and $L=K_X+\Delta+H$ be  a nef divisor  for some ample $\QQ$-divisor $H$ such that 
$$L^{\perp}\cap \overline{NE}(X/U)=R.$$
Assume $L$ is big over $U$, then it follows from \cite{Keel99}
that $L$ is relatively EWM over $U$ with a birational contraction $f:X\to Z$ to an algebraic space.

By Theorem \ref{t-main1}, we obtain the contraction theorem, i.e., we can run  minimal model program (MMP) in the original sense. 
\begin{thm}\label{c-mmp}
Under the above notation. Assume $k$ is an algebraically closed field of characteristic $p>5$. Assume $h:(X,\Delta)\to U$ to be a klt pair projective over a quasi-projective variety $U$.  Then 
\begin{enumerate}
\item $Z$ is a quasi-projective variety, $\rho(X)-\rho(Z)=1$, and 
\item if  $f$ is small, then there exists  a flip $f^+:X^+\to Z$.
% (see \ref{d-gen} for the definition) with the extra properties that:  $K_{X^+}+\Delta^+$ is ample on any fiber of $f^+$ and $\rho(X)=\rho(X^+)$.
\end{enumerate}
\end{thm}

Therefore, we see  the concept of `generalized MMP' is not needed anymore. But it is still  conceptually  important as an intermediate step.  In other words, unlike in characteristic 0, where  the base point free theorem was established much earlier than the existence of flips, to prove the base point free theorem in characteristic $p>0$, the experience suggests that we might need to first establish everything for the MMP of `special type' in the sense of Shokurov. 

The note is organized in the following way: we discuss some preliminary results in Section \ref{s-pre}, especially we state Keel's theorems in \cite{Keel99} in the relative case. Then in Section \ref{ss-gmmp}, we survey the results in \cite{HX13,Birkar13} on running a generalized MMP which we will need. Finally, we finish the proof of Theorem \ref{t-main1} in Section \ref{s-proof}.

%The relative case follows easily, which we include here for future reference.
%\begin{cor}\label{c-rel}
%The relative version of Theorem \ref{t-main1} holds for $X$ being projective over a quasi-projective variety $Z$. 
%\end{cor}

\bigskip

\noindent{\bf Acknowledgement:}  We would like to thank Caucher Birkar for the communications. We started to write the preprint after Caucher Birkar sent us the first draft of his work \cite{Birkar13}, where the generalized MMP was extended to arbitrary coefficients. Later when we finished this preprint, he sent us his second draft in which the main Theorem \ref{t-main1} is proved independently by a different method. We also want to thank Bhargav Bhatt, Hiromu Tanaka for helpful suggestions and the anonymous referee for many useful remarks on the exposition. CX is partially supported by the Chinese grant `Recruitment Program of Global Experts' and Qiushi Outstanding Young Scholarship.

\bigskip

\noindent {\bf Notation and Conventions}: We follow the notation as in \cite{KM98}. For any divisor $\Gamma$ on a normal variety $X$ and a birational map $X\dasharrow X'$ to a normal variety $X'$, we will denote by $\Gamma_{X'}$ to be its birational transform  on $X'$. For $f:X\to Z$ a birational morphism, we denote by ${\rm Ex}(f)$ or ${\rm Ex}(X/Z)$ to be the exceptional set.

\section{Preliminary}\label{s-pre}

In this section, we discuss some backgrounds. We note that resolution of singularities is known in dimension 3 in arbitrary characteristic (see \cite{Ab98, Cutkosky09, CP08, CP09}). 

\subsection{Basic facts}\label{s-basic}
%Let $X\dasharrow Y$ be a birational map of algebraic spaces which are proper over an algebraic space $Z$. For a $\mathbb{Q}$-line bundle $L$ on $X$, we say it is {\it $f$-trivial} if  $L_Y=f_*L$ is $\mathbb{Q}$-Cartier and for a common resolution $p:W\to X$ and $q: W\to Y$, $p^*L\sim_{\mathbb{Q}} q^*L_Y$.

\begin{lem}[Negativity Lemma]\label{l-neg} 
Let $f:X\to X'$ be a proper birational morphism from a quasi-projective normal variety to a normal algebraic space. Let $E$ be an effective $f$-nef $\mathbb{Q}$-divisor which is exceptional over $X'$, then $E=0$.   
\end{lem}
\begin{proof}It suffices to observe that the usual argument  (see \cite[3.39]{KM98}) of cutting $X$ by hyperplanes still holds in this situation.
\end{proof}

\begin{lem}\label{l-snc}
Let $(X,D)$ be a simple normal crossing pair, and $|H|$ a very ample linear system. Then there exists a prime divisor $A\in |H|$, such that $(X,D+A)$ is simple normal cr ossing.
\end{lem}
\begin{proof}Write $D=\sum^N_{i=1} D_i$, where $D_1,...,D_N$ are prime divisors. For any $I\subset \{1,2,...,N\}$, let  
$$D_I^j\subset \bigcap_{i\in I}D_i$$
be an irreducible component of positive dimension. We denote by  $|H|_{{D^j_I}}$ the restricted linear system, i.e., the one corresponds to the image of the restriction map 
$$h^j_I: H^0(X,H)\to   H^0({D^j_I},H|_{D^j_I}).$$
By Bertini theorem (see \cite[II 8.18]{Hart}), we know that there exists an open subset 
$V^j_I\subset |H|_{{D^j_I}}$ such that  any element $A^j_I\in V^j_I$ corresponds to a smooth hypersurface of $D_I^j$. For any $D^j_I$, by definition
$$f^j_I: |H|\setminus Z^j_I\to  |H|_{{D^j_I}}$$
is surjective, where $Z^j_I$ is the proper linear subsystem given by the kernel of $h^j_I$. Thus, we can take 
$$H\in \bigcap_{I,j}(f^j_I)^{-1}(V^j_I).$$
\end{proof}

\begin{prop}[{Bertini Theorem}]\label{p-bertini}
 Let $(X,\Delta)$ be a quasi-projective klt pair and $A$ an ample $\mathbb{Q}$-divisor. Assume $f:Y\to (X,\Delta)$ is a log resolution, such that the exceptional locus ${\rm Ex}(f)$ supports a relatively anti-ample divisor $E$. Then there  exists a $\mathbb{Q}$-divisor $\Delta'$, such that $\Delta'\sim_{\mathbb{Q}}\Delta+A $ and $(X,\Delta')$ is klt. Futhermore, in dimension three, we can show the same result only assuming $A$ to be big and nef. 
\end{prop}
\begin{proof}By negativity lemma (see \cite[3.39]{KM98}), we know that $E \ge 0$. Since it is relatively anti-ample, we have ${\rm Supp}(E)={\rm Ex}(f)$.

Write $f^*(K_X+\Delta)=K_Y+\Delta_Y$ ($\Delta_Y$ may not be effective). Then by the assumption we know  that $ (Y,{\rm Supp }(\Delta_Y+{\rm Ex}(f)))$ is simple normal crossing. Since $(X,\Delta)$ is klt, if we write $\Delta_Y=\sum_i a_iE_i$, then $a_i<1$. There exists a sufficiently small $\epsilon>0$ such that $f^*A-\epsilon E$ is ample on $Y$ and the coefficients of $\Delta_Y+\epsilon E$ are strictly less than 1.

Then we can choose $H\sim_{\mathbb{Q}}f^*A-\epsilon E$, such that  the coefficients of $\Delta_Y+\epsilon E+H$ are strictly less than 1 and $ (Y,{\rm Supp }(\Delta_Y+E+H))$ is simple normal crossing (cf. Lemma \ref{l-snc}). Thus for any divisorial valuation $v$ of $K(X)$, its  discrepancy $a(v,Y, \Delta_Y+\epsilon E+H)>-1$.
Therefore, we can choose $\Delta'=f_*(\Delta_Y+\epsilon E+H).$

In dimension three, log resolution of any pair exists and as it is obtained by a sequence of blow ups with smooth centers, we know that the exceptional locus of such a log resolution always supports a relatively anti-ample divisor $E$. We can write $A\sim_{\mathbb{Q}}A'+G$ for an ample $\QQ$-divisor $A'$ and a sufficiently small effective $\QQ$-divisor $G$, such that
$(X,\Delta+G)$ is still klt. Then we only need to apply the case of $A$ being ample. 
 \end{proof}

\begin{prop}\label{p-curve}
Let $X$ be a connected variety. If $\pi:X_1\to X$ is a finite flat morphism, and $\pi^*L$ is trivial for a line bundle $L$ on $X$, then $L$ is torsion in ${\rm Pic}(X)$.
\end{prop}
\begin{proof}Since by our assumption $\pi_*N$ is locally free on $X$ for any line bundle $N$ on $X$, then we can define (see \cite[6.5.1, hypothesis I]{EGA2})
$${\rm Nm}: {\rm Pic}(X_1)\to {\rm Pic}(X),$$ with the property that ${\rm Nm}\circ \pi^*(L)=L^{\otimes d}$, where $d=\deg(\pi)$. As ${\rm Nm}(\mathcal{O}_{X'})=\mathcal{O}_X$, we conclude that $L$ is a torsion in ${\rm Pic}(X)$. 
\end{proof}

\subsection{Relative version of Keel's theorem}
For the relative setting of proper morphism between quasi-projective varieties, we can take a projectivization and then work in the category of projective varieties. For our purpose, we would like to directly treat it,  which requires us to generalize Keel's theorems to the relative case for $X\to Z$ a proper morphism between quasi-projective varieties. The argument is essentially verbatim. 

\begin{defn}Let $f:X\to Z$ be a proper morphism between quasi-projective schemes. Let $L$ be a relatively nef line bundle. For a  subvariety $W\subset X$,  we denote the Stein factorization $W\to V\to Z$. We say that $W$ is {\it relatively exceptional} (for $L$) if  $L^{\dim W_{\eta}}|_{W_{\eta}}=0$, where $\eta$ is the generic point of $V$. We denote by $\mathbb{E}(L/Z)$  the Zariski closure of the union of all relatively exceptional varieties (with reduced structure).
\end{defn}

\begin{defn} With the above notation, we say that $L$ is {\it endowed with a map over $Z$} (or {\it EWM} over $Z$), if there is a morphism  $g:X\to Y$ to an algebraic space over $Z$ with the property that a subvariety $W$ is contracted by $g$ if and only if it is relatively exceptional.
\end{defn}

Then we have the following relative version of Keel's theorem.

\begin{prop}\label{p-relkeel}
Let $f:X\to Z$ be a proper morphism between quasi-projective schemes. Let $L$ be a nef line bundle on $X$. Then $L$ is EWM (resp. semi-ample) over $Z$ if and only if $L|_{\E(L)}$ is EWM (resp. semi-ample) over $Z$.
\end{prop}
\begin{proof} The proof just follows Keel's original one. We will briefly sketch it for the reader's convenience. 
First, as the natural morphism $X_{\rm red}\to X$ can factor through the iterated geometric Frobenius
$$F^{q}:X \to X^{(q)}_{\rm red} \to X^{(q)}$$
for some $q=p^r$ $(r\gg 0)$ (see \cite[6.6]{Kol95}) and $(F^q)^*(L)=L^{\otimes q}$, we can assume that $X$ is reduced. 

We write $X=X_1\cup X_2$, where $X_1={\mathbb{E}(L/Z)}$ and $X_2$ is the union of components of $X$, on which the restriction of  $L$ is big. If $X_2\neq X$, from the induction, we can assume that $L|_{X_2}$ is EMW (resp. semi-ample). Now the proof of \cite[2.6]{Keel99} also works in this relative setting. 

Therefore, we can assume that $L$ is big on $X$, i.e., 
$$nL\sim_{{Z}}A+D$$ for some $n\in \mathbb{N}$, an ample divisor $A$ and an effective divisor $D$. As the proof of \cite[1.6]{Keel99}, we can use \cite[3.1, 6.2]{Artin70} to construct an algebraic space $Z$ such that $f:X\to Z$ is the endowed map for $L$.

In the case of $L|_{\E(L)}$ being semi-ample, the same argument as in \cite[1.10]{Keel99} implies that $L$ is also semi-ample.
\end{proof} 
\begin{rem}When $Z$ is an algebraic space, the part to conclude that $f$ is EWM if $f|_{\E(L)}$ is EMW follows from a standard descent argument. But we do not know whether the statement for semi-ampleness holds or not. 
\end{rem}

We need the following  statement. 
\begin{cor}\label{c-des}Let $X$ and $Y$ be quasi-projective normal varieties which are projective over a quasi-projective $Z$ with a birational morphism $f:X\to Y$ over $Z$. Let $\Delta$ be a $\mathbb{Q}$-divisor on $X$ such that $(X,\Delta)$ is a dlt pair.
 Let  $ S\subset \lfloor\Delta\rfloor$ be a normal prime divisor such that  $-(K_X+\Delta)|_S$ is $f$-ample and $C\cdot S<0$
for any contracted curve $C$. If $L$ is a line bundle on $X$ such that $L\cdot C=0$ for any contracted curve $C$, then
$L\sim_{\mathbb{Q}}f^*L_Y$ for some $\mathbb{Q}$-line bundle $L_Y$ on $Y$. 
\end{cor}
\begin{proof}As we assume $S$ to be normal, we can write $(K_X+\Delta)|_S=K_S+{\rm Diff}_S\Delta.$ Thus
%Let $H$ be an ample $\mathbb{Q}$-divisor which is  $\mathbb{Q}$-linearly equivalent to $-K_X-\Delta$ over $Y$. Then 
$$L|_S-S-{\rm Diff}_S\Delta$$ is ample over $Y$ and $(S, {\rm Diff}_S\Delta)$ is dlt. 
Since $S\to f(S)$ has connected fibers, we know that the normalization morphism $f(S)^{\rm n}\to f(S)$ is finite and universal homeomorphism. Therefore, it follows from the abundance theorem for log canonical surfaces (see e.g \cite[15.2]{Tanaka12}) that
$L|_S\sim_{\mathbb{Q},f(S)^{\rm n}} 0$, which implies $L|_S\sim_{\mathbb{Q},f(S)} 0$ by \cite[6.6]{Kol95}.
Thus by Proposition \ref{p-relkeel},
$L$ is $\mathbb{Q}$-linearly equivalent to 0 over $Y$, i.e.
$$L\sim_{\mathbb{Q}}f^*(L_Y)$$
for some $\mathbb{Q}$-line bundle on $L_Y$.
\end{proof}

\begin{rem}\label{r-relkeel}
We also remark that Keel's Cone Theorem \cite[5.5.2]{Keel99} holds for the relative setting $X/U$, where $X\to U$ is a projective morphism to a quasi-projective variety. His original proof can be directly applied without any change. 
\end{rem}

%\subsection{${\rm Pic}^0$ of a reduced curve}
%In this subsection, we show the following well known statement. For the terminology and background, see \cite[Section 7.5]{Liu02}.
%\begin{prop}\label{p-curve}
%Let $C$ be a connected  proper reduced curve with only ordinary multiple points. If $\pi:C_1\to C$ is a finite \'etale morphism, then the induced morphism 
%$$\pi^*: {\rm Pic}^0(C)\to {\rm Pic}^0(C_1)$$
%is injective.
%\end{prop}
%\begin{proof} Let $C^n$ (resp. $C^n_1$) be the normalizations of $C$ (resp. $C_1$). Then $C^n_1\to C^n$ is finite \'etale.
% We have an exact sequence (cf. proof of \cite[7.5.18]{Liu02})
%$$0\to k^*\to (k^*)^c\to \mathcal{G}(C)\to {\rm Pic}^0(C)\to {\rm Pic}^0(C^n),$$
%where $\mathcal{G}(C)$ is the skyscraper sheave supported on the singularities. We have a similar exact sequence for $C_1$, and we easily see 
%$$ \mathcal{G}(C)\to \mathcal{G}(C_1)\qquad{\mbox{and}}\qquad {\rm Pic}^0(C^n)\to{\rm Pic}^0(C_1^n)$$
%are injective. Then a diagram chase shows that 
%$$ {\rm Pic}^0(C)\to{\rm Pic}^0(C_1)$$ 
%is injective. 
%\end{proof}
%This statement is indeed true for any proper reduced  curve $C$. But for our purpose, we only need it for curves with ordinary multiple points.  

\section{Running generalized MMP for threefolds}\label{ss-gmmp}  

In this section, we assume $k$ is an algebraically closed field with ${\rm char}(k)>5$.  We provide a short sketch of the proof on the results from \cite{HX13,Birkar13} which we will need.
Since we still need to prove the base point free theorem now, a priori we can not run MMP in the original sense. In \cite{HX13}, a notion called {\it generalized MMP} was invented. 
In fact, for a three dimensional klt pair $(X,\Delta)$,  using Shokurov's idea in \cite{Shokurov92} of reducing MMP to `special' MMP, it is proved  (see \cite{HX13,Birkar13})  that we can always run a generalized MMP over $Z$ in one of the following two cases:
\begin{enumerate}
\item[(I)] when $X$ is projective over some quasi-projective variety $U$, $X\to Z$ is the relative endowed map for some big and nef divisor $L$ over $U$; or
\item[(II)] when $Z$ is quasi-projective.
\end{enumerate}

\begin{defn}\label{d-gmmp}
For a dlt pair $(X,\Delta)/U$, let $S=\lfloor \Delta \rfloor$. We call an  extremal birational contraction $X\to Y$ induced by a $(K_X+\Delta)$-negative ray $R$ of $NE(X/U)$ to be of  {\it  special type}, if $R\cdot S_i<0$ for some component $S_i\subset S$.
\end{defn}

If follows from \cite{Keel99} (and Proposition \ref{p-relkeel})  that if $(X,\Delta)$ is a 3-dimensional dlt pair projective over a quasi-projective variety $U$ and $K_X+\Delta\sim_{\mathbb{Q},U} E\ge0 $, then there exists an ample divisor $H$ such that
$$(K_X+\Delta+H)^{\perp}\cap  \overline{NE}(X/U)=R.$$
If we assume $X$ is $\mathbb{Q}$-factorial, then each component $S_i$ of $\lfloor \Delta \rfloor$ is normal by \cite[4.1]{HX13}. Furthermore, if $R\cdot S_i<0$, Proposition \ref{p-relkeel} and the arguments used in Corollary \ref{c-des} imply that the extremal contraction $X\to Z$ exists and we get $Z$ as a quasi-projective variety.  If $X\to Z$ is small, then it follows from \cite[Theorem 1.1]{HX13} that the flip $X^+\to Z$ exists. Thus, if we start from a $\mathbb{Q}$-factorial dlt three dimensional pair $(X,\Delta)/U$, and we assume that we run a sequence of special type MMP such that each time the contraction is birational, e.g., 
$X\to U$ is birational, then we will have a sequence of models
$$(X,\Delta)=(X_0,\Delta_0)\dasharrow (X_1,\Delta_1)\dasharrow (X_2,\Delta_2)\cdots\dasharrow (X_n,\Delta_n),$$
where Corollary \ref{c-des} can be applied to show that each model $X_i$ is $\mathbb{Q}$-factorial. 
Special termination (see \cite[4.2.1]{Fujino07} or \cite[4.7]{Birkar13}), which holds in this case, implies that the MMP will end with a relative minimal model $X_n$ over $Z$.

Let first introduce the definition of a generalized MMP.  

\begin{defn}\label{d-gen}
Let $(X,\Delta)$ be a klt pair which is projective over a quasi-projective variety $U$  and $R$ an extremal ray.  Let $f:X\to Z$ be the birational extremal contraction corresponding to a $(K_X+\Delta)$-negative extremal ray 
$$R=\mathbb{R}_{\ge 0}[C] \subset  \overline{NE}(X/U)$$ with the target space possibly being an algebraic space, i.e., $Z$ is a normal algebraic space, $f:X\to Z$ is the endowed map (by Proposition \ref{p-relkeel}) for a big and nef line bundle $L=K_X+\Delta+H$ satisfying 
$$L^{\perp}\cap \overline{NE}(X/U)=R.$$ 
We say that $(X^+,\Delta^+)$ is  {\it a step of the generalized MMP} if $X\dasharrow X^+$ is birational and $\Delta=\phi^+_*\Delta$ and $K_{X^+}+\Delta^+$ is nef over $Z$.
When $\phi^+:X\dasharrow X^+$ is isomorphic in codimension one, we call it a {\it generalized flip}.

%We say  {\it we can run a generalized MMP} if for $(X,\Delta)/U$ and $R$ as above which induces a birational contraction $X\to Z$, a step of generalized MMP $X^+$ always exists. 
\end{defn}

\begin{thm}[\cite{HX13,Birkar13}]
Let $k$ be an algebraically closed field with ${\rm char}(k)>5$. Assume $X$ to be a $\mathbb{Q}$-factorial threefold. Let $(X,\Delta)\to U$ be a  klt pair projective over a quasi-projective variety $U$.   Assume 
$f:X\to Z$ is given as in Definition \ref{d-gen}, then a step of generalized MMP always exists, i.e., we can always find $X^+$ as in Definition \ref{d-gen}.%Then  we can always run a generalized MMP. 
\end{thm}

\begin{proof}
%Let $Z$ be the EWM defined as in  Definition \ref{d-gen}.

 If the coefficients of $\Delta$ are contained in $\{\frac{n-1}{n}|n\in \mathbb N\}\cup \{1\}$, then this follows from \cite[5.6]{HX13}. We remark that since we have relative Cone Theorem \ref{r-relkeel}, relative Contraction Theorem \ref{p-relkeel} and special flip, then the argument  in \cite[5.6]{HX13}  holds in this relative setting. 

In  general, we repeat the argument of \cite{Birkar13}, which reduces the general case to to the above case of standard coefficients as following:

Write $\Delta=\sum^{m}_{i=1}a_i\Delta_i$, where $\Delta_1,...,\Delta_m$ are distinct prime divisors. We will show that if we can run a generalized MMP for any klt pair $(X',\Delta')$ birational over $Z$ with coefficients in 
$$ I_0:=\{\frac{n-1}{n}|n\in \mathbb N\}\bigcup \{a_1,...,a_{m-1}\}\bigcup \{1\} ,$$
then we have a step of the generalized MMP for $K_X+\Delta$, hence the theorem follows from the induction.

As part of our induction, we also assume that a step of the generalized MMP for the coefficients in $I_0$ is given by
a  sequence of birational models,
$$X=X_0\dasharrow X_1\dasharrow X_2....\dasharrow X_n=X^+,$$
where $X_i\dasharrow X_{i+1}$ is one of the following two operations:
\begin{enumerate}
\item $X_{i+1}\to (X_i,\Delta_i)$ is  a log resolution for some divisor $\Delta_i$, or
\item $X_{i}\dasharrow X_{i+1}$ is a sequence of special MMP with standard coefficients, which can be run by \cite{HX13}. 
\end{enumerate}

Write $\Delta=\Delta_1+\Delta_2$, where $\Delta_1=\sum^{m-1}_{i=1}a_i\Delta_i$ and $\Delta_2=a_m\Delta_m$. 

 Assume $\Delta_m\cdot R\ge0$. Then we know that 
$\Delta_2\equiv_Z-t(K_X+\Delta_1)$ for some $t\in [0,1).$ By induction, we can run generalized MMP for $(X,\Delta_1)$ over $Z$, which provides a  sequence of birational models
$$X=X_0\dasharrow X_1\dasharrow X_2\dasharrow....\dasharrow X_n=X^+.$$
Using induction, we also know $X_i\dasharrow X_{i+1}$ where $X_i\dasharrow X_{i+1}$ is either a log resolution for some divisor $\Delta_i$, or
 a sequence of special MMP with standard coefficients.

In particular, by repeatedly applying Corollary \ref{c-des}, we see that  if there is a $\mathbb{Q}$-line bundle $L$ on $X$, such that $L\equiv_Z 0$, then there is a  $\mathbb{Q}$-line bundle $L^+$ on $X^+$, such that if we pull back $L$ and $L^+$ to a common resolution, we get two $\mathbb{Q}$-line bundles which are relatively $\mathbb{Q}$-linear equivalent to each other. Thus 
$$K_{X^+}+(\Delta_1)_{X^+}+(\Delta_2)_{X^+}\equiv (1-t)(K_{X^+}+(\Delta_1)_{X^+}),$$
which is nef. 

Now we assume $\Delta_2\cdot R<0$. We can apply the argument in \cite[5.6]{HX13} by choosing $T=\Delta_m$. Then we get a model $W$ as there such that $K_W+(\Delta_1)_W+E_W+T_W $ is nef over $Z$ where $E_W$ is the divisorial part of the exceptional locus ${\rm Ex}(W/Z)$. Now we run generalized MMP of $K_W+(\Delta_1)_W+E_W$ with scaling of $T_W$ over $Z$, which exists by our induction assumption. As argued in \cite[5.6]{HX13}, this is of special type. In particular, it terminates. By the definition of MMP with scaling, we know that it provides a model
$$W\dasharrow W_r=X^+,$$
such that $$K_{X^+}+(\Delta_1)_{X^+}+E_{X^+}+a_mT_{X^+} $$ is nef, where $E_{X^+}$ is the push forward of $E_W$ on $X^+$. Then the negativity lemma implies that $E_{X^+}=0$. 

In each case, we also see that the map $X\dasharrow X^+$ can be decomposed into maps of type (1) and (2).
\end{proof}

%So starting from a three dimensional klt pair $(X,\Delta)$, replacing $X$ by its $\mathbb{Q}$-factorialization, we can assume $X$ is $\QQ$-factorial. 

%\begin{defn}\label{d-flip}A generalized flip is called {\it a flip}
%if $X^+$ satisfies the properties that $K_{X^+}+\Delta^+$ is ample on any fiber of $f^+$ and there exists a divisor $L^+$ such that the pull back of $L$ and $L^+$ to a common resolution are $\mathbb{Q}$-linearly equivalent. 
%\end{defn}

%\begin{lem}For a klt pair three dimensional pair $(X,\Delta)$, if a flip exists, then it is unique. 
%\end{lem}
%\begin{proof}By the cone theorem, we know that 
%We claim for sufficiently large $m$, $mL^++K_{X^+}+\Delta^+$ is ample. Thus 
%$$X^+={\rm Proj}\bigoplus_m H^0(X,m(K_X+\Delta+L)).$$
%\end{proof}
\begin{rem}\label{r-flop}

As we already pointed out in the argument, by Corollary \ref{c-des}, we know that if we have a $\mathbb{Q}$-line bundle $L$ on $X$, such that $L\equiv_Z 0$, then there is a  $\mathbb{Q}$-line bundle $L^+$ on $X^+$, such that the pull backs of $L$ and $L^+$ to a common resolution are  $\mathbb{Q}$-linear equivalent to each other.
\end{rem}

As a standard consequence of running generalized MMP and special termination, we know that
\begin{lem}\label{l-dlt}Assume $k$ is an algebraically closed field with ${\rm char}(k)>5$.
If $(X,\Delta)$ is a log canonical three dimensional pair, then a $\mathbb{Q}$-factorial dlt modification exists. 
\end{lem}

If we assume $K_X+\Delta$  is effective over $U$,  the termination of generalized MMP is proved in \cite{Birkar13} following the idea in \cite{Birkar07}. It is shown that the termination of a sequence of generalized flips in this case is implied by 3-dimensional ascending chain condition (ACC) of log canonical thresholds. Then ACC of three dimensional log canonical thresholds is a corollary of 2-dimensional global ACC (cf. \cite[18.21]{Kollar91} or \cite[Section 5]{HMX}). The latter was proved by Alexeev in \cite{Alexeev94}. 

To summarize, we have the following result which we need later.
\begin{thm}[\cite{HX13,Birkar13}]\label{t-mmp}
Assume $k$ to be an algebraically closed field with ${\rm char}(k)=5$. Let $(X,\Delta)$ be a klt  three dimensional quasi-projective pair with a proper morphism to $Z$ such that one of the the following cases holds: 
\begin{enumerate}
\item[(I)]  $X$ is projective over some quasi-projective variety $U$, $X\to Z$ is the relative endowed map for some big and nef divisor $L$ over $U$; or
\item[(II)] $Z$ is quasi-projective and $K_X+\Delta\sim_{\mathbb{Q},Z}E\ge 0.$
\end{enumerate}
Then we can run a generalized MMP of $(X,\Delta)$ over $Z$ to obtain a minimal model $(X^{\rm m},\Delta^{\rm m})$ of $(X,\Delta)$ over $Z$. 

Furthermore, if $L$ is a $\mathbb{Q}$-line bundle on $X$ such that $L\equiv_Z 0$, then there exists a $\mathbb{Q}$-line bundle $L_{X^{\rm m}}$ on $X^{\rm m}$ such that the pull backs of $L$ and $L_{X^{\rm m}}$ to a common resolution is $\mathbb{Q}$-linearly equivalent. 
 \end{thm}

\section{Proof}\label{s-proof}

In this section, we always assume $k$ is an algebraically closed field with ${\rm char}(k)>5$. Let $(X,\Delta)$  be a klt pair  and $L$ be a  big and nef  $\mathbb{Q}$-divisor such that $K_X+\Delta+L$ is also big and nef.  By Bertini type Theorem \ref{p-bertini}, after replacing $\Delta+L$ by $\Delta'$, we indeed can assume $K_X+\Delta$ is big and nef, and we aim to show that $K_X+\Delta$ is semi-ample.  Let $X\to Z$ be the endowed map for $K_X+\Delta$ provided by \cite{Keel99}.

\begin{lem}
Assume the above notation. There is a birational contraction $f:X\dasharrow Y$ over $Z$ such that $(Y,\Delta_Y=f_*(\Delta))$ is klt and $\dim\mathbb{E}(K_Y+\Delta_Y)=1$.
\end{lem}
\begin{proof}
Let $S$ be the sum of divisorial components $\dim\mathbb{E}(K_X+\Delta)$.   Let $\epsilon>0$ be sufficiently small such that $K_X+\Delta_X+\epsilon S$ is klt. By Theorem \ref{t-mmp}, we can run generalized MMP of $(X,\Delta+\epsilon S)$ over $Z$ to obtain a  birational model $f:X\dasharrow Y$ such that $K_Y+\Delta_Y+\epsilon S_Y$ is nef over $Z$, where $\Delta_Y$ and $S_Y$ are the push forwards of $\Delta$ and $S$. 

From the assumption that $S\subset \dim\mathbb{E}(K_X+\Delta)$, we know that $S$ is exceptional over $Z$.  By negativity Lemma \ref{l-neg}, we know that the components of $S$ are precisely the divisors contracted by $X\dasharrow Y$.
\end{proof}

\begin{proof}[Proof of Theorem \ref{t-main1}]Replacing $X$ by $Y$, we can assume that  $\dim\mathbb{E}(K_X+\Delta)=1$. Replacing $X$ by its $\mathbb{Q}$-factorialization, then we can assume $X$ it is $\mathbb{Q}$-factorial.  By Keel's theorem,  if we denote by $L=K_X+\Delta$, then it suffices to show $L|_{\mathbb{E}(K_X+\Delta)}$ is semi-ample, i.e.,  for any closed point $p\in Z$, assume $C=f^{-1}(p)$ is a connected curve, then $L|_C$ is semi-ample. 

We first assume $p\in Z$ is contained in an quasi-projective open neighborhood of $Z$. After possibly replacing $Z$ by a neighborhood, we can assume $Z$ is quasi-projective and $C$ is the exceptional locus ${\rm Ex}(X/Z)$.

The following construction is standard in characteristic 0, which is a combination of running MMP after using X-method. 
\begin{prop}Assume $Z$ to be quasi-projective. 
 Locally over $Z$, we can find an effective $\mathbb{Q}$-divisor $H\sim_{\mathbb{Q}} L$ such that  if we define 
$$c={\rm lct}(X,\Delta;H)$$ to be the log canonical threshold along $C$, then
\begin{enumerate}
\item $(X,\Delta+cH)$ is plt along $C$.
\item There exists a $\mathbb{Q}$-factorial model $W/Z$ and a prime divisor $E$ on $W$, such that $(W,E+\Delta_W+cH_W)$ is plt 
where $\Delta_W$ and $H_W$ are the birational transforms $\Delta$ and $H$ on $W$, and if $p:U\to X$ and $q:U\to W$ is a common resolution, then
$$p^*(K_X+\Delta+cH)=q^*(K_W+\Delta_W+cH_W+E).$$ 
\item $-E_W$ is nef over $Z$. 
\end{enumerate}
\end{prop}
\begin{proof} 

As $L$ is big, we can write it as $A+B$ where $B\ge 0$ and $A$ is ample. It is  standard that after using $A$ to tie-and-break (see e.g. \cite[3.2.3]{HMX}), we can find a $\mathbb{Q}$-divisor $H\sim_{\mathbb{Q}}L$, such that for
$$c={\rm lct}(X,\Delta;H)$$
the log canonical threshold along $C$, the pair $(X,\Delta+cH)$ is log canonical with precisely one divisorial valuation $v$ such that the discrepancy $a(v;X,\Delta+H)=-1$.

 After possibly shrinking $Z$, applying Lemma \ref{l-dlt}, we can find $h:V\to X$ which is a $\mathbb{Q}$-factorial  dlt modification of $(X,\Delta+cH)$, i.e., $V$ is $\mathbb{Q}$-factorial and $(V,\Delta_V+cH_V+E)$ is dlt where $E$ being the divisorial part of ${\rm Ex}(h)$ which corresponds to the valuation $v$. Hence
 $$h^*(K_X+\Delta+cH)=K_V+\Delta_V+cH_V+E.$$ 
Then by Theorem \ref{t-mmp}, we can run a generalized MMP for 
 $$K_V+\Delta_V+cH_V\equiv_{Z}-E$$ over $Z$ as in Section \ref{ss-gmmp}, and we end up with a relatively minimal model $f_W: W\to Z$.  Since each step of the generalized MMP  is $E$-positive, $E$ is not contracted in $V\dasharrow W$.
 Then we define $E_W$ to be the push forward of $E$ on $W$.  \end{proof}

It follows from Remark \ref{r-flop} that $K_W+\Delta_W+cH_W+E_W$ is numerically trivial over $Z$. By \cite[4.1]{HX13} and the fact that $W$ is $\mathbb{Q}$-factorial, we know that $E_W$ is normal as $(W,E_W)$ is plt. Thus, restricting on $E_W$, we have that
$$K_W+\Delta_W+cH_W+E_W|_{E_W}=K_{E_W}+{\rm Diff}_{E_W}(\Delta_W+cH_W)$$ is numerically trivial and 
$(E_W, {\rm Diff}_{E_W}(\Delta_W+cH_W))$ is klt. Therefore it follows from the abundance for surfaces(cf. e.g. \cite{Tanaka12}) that 
$$K_{E_W}+{\rm Diff}_{E_W}(\Delta_W+cH_W)\sim_{\mathbb{Q}}0.$$

 Since $-E_W$ is nef over $Z$ we know that 
 $$E_W=f_W^{-1}(f_W(E_W))\supset {\rm Ex}(f_W),$$ hence  $K_W+\Delta_W+cH_W+E_W$ is $\mathbb{Q}$-linearly trivial over $Z$ by Proposition \ref{p-relkeel}, i.e., 
 $$K_W+\Delta_W+cH_W+E_W\sim_{\mathbb{Q}} f^*_W(A)$$
 for some $\mathbb{Q}$-line bundle $A$ on $Z$.
By Theorem \ref{t-mmp}, we know that
$$ K_Y+\Delta_Y+cH_Y+E\sim_{\mathbb{Q}}  K_X+\Delta+cH\sim_{\mathbb{Q}} f^*(A).$$
Thus $(K_X+\Delta)|_C\sim_{\mathbb{Q}}0$.

Now let us consider the general case. For $p=f(C)\subset Z$, let $\pi:\tilde{Z}\to Z$ be the \'etale map from a quasi-projective \'etale neighborhood of $p$. Let 
$$\pi^X: X\times_Z\tilde{Z}\to X$$ be the \'etale morphism. 
We denote by $C_1=(\pi^X)^{-1}(C)$ which is finite \'etale over $C$. 
%Let  $\phi:\tilde{C}\to C$ be the  reduced curve with only ordinary multiple points associated to $C$, which is homeomorphism to $C$ (see \cite[7.5.12]{Liu02}). Thus to show $L|_C$ is semi-ample, it suffices to show that $\phi^*L $ is semi-ample on $\tilde{C}$. Let $\tilde{C}_1=\tilde{C}\times_C C_1$ which will be the curve with only ordinary multiple points associated to $C_1$. For some $m\in \mathbb{N}$, the pull back $(\pi^X)^*(mL)|_C$ is trivial on $C_1$ by the established case when $Z$ being quasi-projective. 
Thus it follows from Proposition \ref{p-curve} that $\phi^*nL$ is trivial for $n$ sufficiently divisible.  
\end{proof}
%%%%%%%%%%%%%%%%
%\begin{proof}[Proof of \eqref{c-mmp}]Let $H$ be an ample $\mathbb{Q}$-divisor such that 
%$$R=(K_X+\Delta+H)^{\perp}\cap \overline{NE(X)}. $$
%Denote by $L=K_X+\Delta+H$, and let $X\to Z$ be the endowed map for $L$ provided by \cite{Keel99}.

%We can run generalized-MMP for $(X,\Delta)$ to get a generalized flip $X^+$, with $L^+$ a $\mathbb{Q}$-line bundle whose pull back to a common resolution is $\mathbb{Q}$-linearly equivalent to the one of $L$. Since $K_{X^+}+\Delta^++H^+$ is big, by cone theorem, we know that
%$$\overline{NE}(X^+)=\overline{NE}(X^+)_{K_{X^+}+\Delta^++H^+>0}+\sum_{i}R_{\ge 0} [C_i].$$

%In particular, we know that for sufficiently $L$

\begin{proof}[Proof of \ref{c-mmp}] (1) The assertion that $Z$ is an algebraic variety follows from Theorem \ref{t-main1}.  In particular, we conclude that $L=f^*L_Z$ for some ample divisor $L_Z$. 
We also have the exact sequence
$$0\to {\rm Pic}(Z)_{\QQ}\to {\rm Pic}(X)_{\QQ}\to \QQ\to 0,$$
where the last morphism is given by the intersection with a curve class $[C]$ in $R$. In fact, if $L'$ is a line bundle on $X$ such that $L'\cdot [C]=0$, then it follows from Cone Theorem, for sufficiently small $\epsilon>0$, $L+\epsilon L'$ is still big and nef over $U$ and 
$$(L+\epsilon L')^{\perp}\cap \overline{NE}(X/U)=R.$$
Thus we conclude that $L+\epsilon L'$ is semiample, whose multiple will be a pull back of an ample $\QQ$-divisor on $Z$.  

For (2), since we know there exists a generalized flip $f':X'\to Z$, such that $K_{X'}+\Delta'$ is nef over $Z$ where $\Delta'$ is the push forward of $\Delta$ to $X'$.
It follows from Cone Theorem that we can find a sufficiently ample divisor $H_Z$ on $Z$, such that $K_{X'}+\Delta'+f'^*H_Z$ is big and nef and all $(K_{X'}+\Delta'+f'^*H_Z)$-trivial curves are vertical over $Z$. So $K_{X'}+\Delta'+f'^*H_Z$ is  base point free by Theorem \ref{t-main1}. Therefore, we can take
$$X^+:={\rm Proj}\bigoplus_{m=0} f'_*\mathcal{O}_{X'}(m(K_{X'}+\Delta'+f'^*H_Z) ),$$
which admits a morphism $f^+:X^+\to Z$ yielding the flip.  
\end{proof}

\begin{rem}The base point free theorem in characteristic 0 is proved for nef $L$ which can be written as  $L\sim_{\QQ}K_X+\Delta+A$ for a klt pair $(X,\Delta)$ and a big and nef  $\QQ$-divisor $A$. However, in characteristic $p>0$, due to the existence of inseparable morphisms, it is still mysterious on how to deal with the case that $L$ not being big. There are partial results $\kappa(L)=1$ or $2$ in \cite{CTX13} with more restrictive assumptions.  
\end{rem}
%We consider the log canonical threshold
%$c={\rm lct}(H;X,\Delta)$. If $c>1$, then $(X,\Delta+H)$ is klt, thus we can apply Theorem \ref{t-main1}.
%If $c\le 1$, as in the proof of Theorem \ref{t-main1}, after first taking the dlt modification and then running generalized MMP, we obtain a model $W$ such that
% $$K_W+\Delta_W+cH_W\equiv_Z -E_W$$
% nef over $Z$. Applying Theorem \ref{t-main1}, we know that there exists $V$, such that $g:W\to V$ is given by sufficiently high multiple 
%of $K_W+\Delta_W+cH_W$. If  $g$ is a small morphism, then $-E_V$ is ample, thus $V\to Z$ is a negative contraction for the plt pair 
%$(K_V+\Delta_V+(c-\epsilon)H_V+E_V)$ and we can apply \ref{c-des} to conclude that $K_V+\delta_V+H_V+E_V\sim_{\mathbb{Q}} 0$.
%\end{proof}

%%%%%%%%%%%%%%%%%%%%%%%%%%%%%%%%%%%%%%%%%%%

\bigskip

\end{document}